\documentclass[11pt]{article}

\usepackage{graphicx, amssymb, latexsym, amsfonts, amsmath, lscape, amscd,
amsthm, color, epsfig, mathrsfs, tikz, enumerate}

\newtheorem{theorem}{Theorem}[section]
\newtheorem{conjecture}[theorem]{Conjecture}
\newtheorem{corollary}[theorem]{Corollary}

\newtheorem{proposition}[theorem]{Proposition}
\newtheorem{question}[theorem]{Question}

\newtheorem{claim}{Claim}

\newcommand\DELETE[1]{}

\newcommand\MONNAESATPB{\textsc{Monotone Not-All-Equal $3$-Satisfiability}}

\newcommand\SSCLIQUEPB{\textsc{Signed Clique Edge-Color Assignment}}

\newcommand\NPclass{\textsf{NP}}

\begin{document}

\title{{\bf Analogous to cliques for $(m,n)$-colored mixed graphs}}
\author{
{\sc Julien Bensmail}$^{a}$, {\sc Christopher Duffy}$^{b,d}$, {\sc Sagnik Sen}$^{b,c}$ \\
\mbox{}\\
{\small $(a)$ CNRS, ENS de Lyon, LIP, \'Equipe MC2, INRIA, Universit\'e Lyon 1, Universit\'e de Lyon, France}\\
{\small $(b)$ Univ. Bordeaux, LaBRI, UMR 5800, F-33400 Talence, France}\\
{\small $(c)$ Indian Statistical Institute, Kolkata, India}\\
{\small $(d)$ University of Victoria, Canada}
}

\date{\today}

\maketitle

% ----------------------------------------------------------------
\begin{abstract}
A $(m,n)$-colored mixed graph is a graph having arcs of $m$ different colors and edges of $n$ different colors. 
A graph homomorphism of a $(m,n)$-colored mixed graph $G$ to a $(m,n)$-colored mixed graph $H$ is a vertex mapping such that if $uv$ is an arc (edge) of color $c$ in $G$, then $f(u)f(v)$ is also an arc (edge) of color $c$. 
The \textit{$(m,n)$-colored mixed chromatic number} $\chi_{(m,n)}(G)$ of a   $(m,n)$-colored mixed graph $G$, introduced by 
 Ne\v{s}et\v{r}il and Raspaud, is the 
order (number of vertices)
of the smallest homomorphic image of $G$.

In this article, we define the analogue of clique  for $(m,n)$-colored mixed graphs. A $(m,n)$-clique is 
a $(m,n)$-colored mixed graph $C$ with $\chi_{(m,n)}(C) = |C|$. We show that almost all  
$(m,n)$-colored mixed graphs are $(m,n)$-cliques. We prove exact bound on the order of the biggest 
outerplanar $(m,n)$-clique and provide upper and lower bounds for 
the order of the biggest 
planar $(m,n)$-clique. Furthermore, we address a particular complexity problem related to $(0,2)$-colored mixed graphs and compare 
the parameters  $\chi_{(1,0)}$ and $\chi_{(0,2)}$. 
 \end{abstract}


\section{Introduction}
The notion of vertex coloring and chromatic number was studied in a  generalized set up by Ne\v{s}et\v{r}il and Raspaud~\cite{raspaud_and_nesetril}
by definining colored homomorphism for $(m,n)$-colored mixed graphs. In this article we will define and study the notion of cliques and clique numbers for the same generalized type of graphs.

A \textit{$(m,n)$-colored mixed graph} $G = (V, A \cup E)$ is a graph $G$ with set of vertices $V$, set of arcs $A$ and set of edges $E$ where 
arcs are colored with $m$ colors and the edges are colored with $n$ colors and the underlying undirected graph is simple.  
Let $G = (V_1, A_1 \cup E_1)$ and $H = (V_2, A_2 \cup E_2)$
be two $(m,n)$-colored mixed graphs. A colored homomorphism of $G$ to $H$ is a function $f : V_1 \rightarrow V_2$ satisfying
$$uv \in A_1 \Rightarrow f(u)f(v) \in A_2,$$ 
$$uv \in E_1 \Rightarrow f(u)f(v) \in E_2,$$
and the color of the arc or  edge linking $f(u)$ and $f(v)$ is the same as the color of the arc or the edge linking $u$ and $v$~\cite{raspaud_and_nesetril}.
  We write $ G \rightarrow  H$ whenever there exists a 
 homomorphism of $ G$ to $ H$.

Given a $(m,n)$-colored mixed graph $G$ let $H$ be a $(m,n)$-colored mixed graph with minimum \textit{order} (number of vertices) such that $G \rightarrow H$. 
Then the order of $H$ is the \textit{$(m,n)$-colored mixed chromatic number} $\chi_{(m,n)}(G)$ of $G$.
The maximum $(m,n)$-colored mixed chromatic number taken over all $(m, n)$-colored mixed
graphs having underlying undirected simple graph $G$ is denoted by  $\chi_{(m,n)}(G)$.
Let $\mathcal{F}$ be a family of undirected simple graphs. Then  $\chi_{(m,n)}(\mathcal{F})$ is the maximum of $\chi_{(m,n)}(G)$ 
taken over all $G \in \mathcal{F}$.

Note that a $(0,1)$-colored mixed graph $G$  is nothing but an undirected simple graph while  
$\chi_{(0,1)}(G)$ is the ordinary chromatic number. 
Similarly, the study of $\chi_{(1,0)}(G)$ 
is  the study of oriented chromatic number which is considered by several researchers in the last two decades (for details please check the recent updated survey~\cite{sopena_updated_survey}).
Alon and Marshall~\cite{Marshall-edgecoloring}  studied the homomorphism of $(0,m)$-colored mixed graphs with a particular focus on $m=2$.

Ne\v{s}et\v{r}il and Raspaud~\cite{raspaud_and_nesetril} showed that $\chi_{(m,n)}(G) \leq k(2m+n)^{k-1}$ where $G$ is a $k$-acyclic colorable graph.
As planar graphs are $5$-acyclic colorable due to Borodin~\cite{Borodinacyclic} , the same authors implied  $\chi_{(m,n)}(\mathcal{P}) \leq 5(2m+n)^4$ for the family $\mathcal{P}$ of planar graphs as a corollary.
This result, in particular, implies $\chi_{(1,0)}(\mathcal{P}) \leq 80$ and $\chi_{(0,2)}(\mathcal{P}) \leq 80$ (independently proved 
before in~\cite{planar80} and~\cite{Marshall-edgecoloring}, respectively). 
It is observed that using similar techniques the same or close upper and lower bounds can be proved for both $\chi_{(1,0)}(\mathcal{F})$ and $\chi_{(0,2)}(\mathcal{F})$ for several graph families (details in Section~\ref{o-vs-s}). This made us wonder if there is an underlying general relation between the two kinds of graph colorings. 
Moreover, note that in the main results, each colored arcs are contributing twice as much  each colored edge. It is natural to wonder if this is 
a universal trend or not. 
In this article we explore this speculation and actually end up proving the opposite, that is, we construct examples of undirected simple graphs for which 
$\chi_{(0,2)}(G) - \chi_{(1,0)}(G)$ is arbitrarily high or low. 

%
%\begin{theorem}\label{theorem difference}
%Given any integer $n$, there exists an undirected graph $G$ such that 
%$\chi_{(0,2)}(G) - \chi_{(1,0)}(G) = n$.
%\end{theorem}

A \textit{$(m,n)$-clique} $C$ is a $(m,n)$-colored mixed graph with $\chi_{(m,n)}(C) = |C|$ (where $|C|$ is the order of the graph $C$). 
The problem is to find a $(m,n)$-clique with  maximum order contained as a subgraph in a $(m,n)$-colored mixed graph $G$. 
That maximum order is the \textit{$(m,n)$-absolute clique number} $\omega_{a(m,n)}(G)$ of  $G$. 
Our definition is a generalization using the intuition of oriented clique~\cite{36} and oriented absolute clique number~\cite{unique_oclique} 
which is the case when $(m,n)= (1,0)$. 
Note that $\omega_{a(0,1)}(G)$ is just the ordinary clique number of a simple undirected graph $G$.

While studying oriented absolute clique number, another related parameter, the oriented relative clique number, arose naturally~\cite{unique_oclique}. 
An analogous parameter seems significant for the study of homomorphisms of $(m,n)$-colored mixed graphs as well. 
A \textit{relative $(m,n)$-clique} $R$ of a $(m,n)$-colored mixed graph $G$ 
is a set of vertices 
such that for any two distinct vertices $u,v \in R$ we have $f(u) \neq f(v)$ for each homomorphism $f :G \rightarrow H$ and for each 
$(m,n)$-colored mixed graph $H$. 
That is, no two distinct vertices of a relative clique can be identified under any homomorphism. 
 The 
\textit{$(m,n)$-reletive clique number} $\omega_{r(m,n)}(G)$ of a $(m,n)$-colored mixed graph $G$ is
 the cardinality of a largest  relative $(m,n)$-clique of $G$. 
For undirected simple graphs the two parameters coincide but they are different for all $(m,n) \neq (0,1)$.
The two parameters for an undirected simple graph and for a family of graphs are  defined similarly as in the case of $(m,n)$-colored mixed chromatic number.

From the definitions it is clear that $\omega_{a(m,n)}(G) \leq \omega_{r(m,n)}(G) \leq \chi_{(m,n)}(G)$ for any $(m,n)$-colored mixed graph $G$. 
We will show that $(m,n)$-cliques are not rare objects by proving that 
almost all $(m,n)$-colored mixed graphs are $(m,n)$-cliques.

%
%\begin{theorem} \label{th assymtotic}
%Almost all $(m,n)$-colored mixed graphs are $(m,n)$-cliques.
%\end{theorem}

We also studied the parameter for the family $\mathcal{O}$ of outerplanar and $\mathcal{P}$ of planar graphs 
and provide exact bound of  $\omega_{a(m,n)}(\mathcal{O}) = \omega_{r(m,n)}(\mathcal{O}) = 3(2m+n) + 1$ for outerplanar graph and 
quadratic upper and lower bound of $3(2m+n)^2+(2m+n)+1 \leq \omega_{a(m,n)}(\mathcal{P}) \leq 9(2m+n)^2 + 2(2m+n) + 2$ 
for planar graphs. 
Note that the present lower and upper bounds for the $(m,n)$-colored mixed chromatic number of these two families 
are $(2m+n)^2+ \epsilon (2m+n)+1 \leq \chi_{(m,n)}(\mathcal{O}) \leq 3(2m+n)^2$
and $(2m+n)^3 +\epsilon (2m+n)^2 + (2m+n) + \epsilon \leq \chi_{(m,n)}(\mathcal{P}) \leq 5(2m+n)^4$ where $\epsilon = 1$ for $m$ odd or 
$m = 0$, and $\epsilon = 2$ for $m > 0$ even~\cite{homsignified}.

%\begin{theorem}\label{n-edge-colored outer clique}
%Let $\mathcal{O}$ be the family of outerplanar graphs. Then $\omega_{a(m,n)}(\mathcal{O}) = \omega_{r(m,n)}(\mathcal{O}) = 3(2m+n) + 1$ for all
%$(m,n) \neq (0,1)$.
%\end{theorem} 
% 
%Furthermore, we will prove lower and upper bounds  for the $(m,n)$-absolute clique number of the family of planar graphs. 
%
%\begin{theorem}
%Let $\mathcal{P}$ be the family of planar graphs. Then $3(2m+n)^2+(2m+n)+1 \leq \omega_{a(m,n)}(\mathcal{P}) \leq 9(2m+n)^2 + 2(2m+n) + 2$ for all
%$(m,n) \neq (0,1)$.
%\end{theorem} 

Given an undirected simple graph it is $NP$-hard to decide if it is underlying graph of an $(m,n)$-clique for $(m,n) = (1,0)$ and $(0,2)$ 
is known~\cite{oclique-complexity}.  We prove a related similar result for an equivalence class of $(0,2)$-colored mixed graph. 

\bigskip

 \textbf{An important  related question:} Naserasr, Rollova and Sopena~\cite{signedhom} recently studied the homomorphism of a particular equivalence class of $(0,2)$-graphs which they called signed graphs. 
Using these notions they managed to reformulate and extend several classical theorems and conjectures including the Four-Color Theorem and the Hadwieger's conjecture. They also defined absolute and relative clique number of signed graphs in a similar fashion. 
We will avoid the definitions related to homomorphisms of signed graphs as it will make this article unnesessarily complicated.
The specific question that we are interested in can be formulated using the notions already defined in this paper. The reader is encouraged to read the above mentioned paper for further details. 
In the ``2$^{nd}$ Autumn meeting  on signed graphs'' (Th\'ezac, France 2013) organized by  Naserasr and Sopena the following question was asked by Naserasr:

\begin{question}\label{q reza}
A $(0,2)$-colored mixed graph is an signed clique if each pair of its non-adjacent vertices 
 is part of a 4-cycle with three edges of the same color while the other edge has a different color. Given an undirected simple graph $G$ what is the complexity of deciding if it is an underlying graph of a signed clique?
\end{question}

Here we prove that this is an  $NP$-complete problem.

\section{On absolute and relative $(m,n)$-clique numbers}\label{k-edge-colored}
First we will characterize the $(m,n)$-cliques.  Let $G$ be an $(m,n)$-colored mixed graph. Let $uvw$ be a 2-path in the underlying simple graph of $G$. Then $uvw$ is a \textit{special 2-path}  of $G$ if one of the following holds:

\begin{itemize}
\item[(i)] $uv$ and $vw$ are edges of different colors,

\item[(ii)] $uv$ and $vw$ are arcs (possibly of the same color),

\item[(iii)] $uv$ and $wv$ are arcs of different colors,

\item[(iv)] $vu$ and $vw$ are arcs of different colors,

\item[(v)] exactly one of $uv$ and $vw$ is an edge.
\end{itemize}

\begin{proposition}\label{prop clique1}
Let $G$ be an $(m,n)$-colored mixed graph. Then two vertices $u,v$ of $G$ are part of a relative clique if and only if either they are adjacent or they are connected by a special 2-path. 
\end{proposition}

\begin{proof}
Let $u,v$ be two vertices of a $(m,n)$-colored mixed graph $G$. If they are not part of any relative clique then there exists a $(m,n)$-colored mixed graph $H$ and a homomorphism $f: G \rightarrow H$ such that $f(u) = f(v)$. 

Firstly, suppose that $u,v$ are adjacent. Then the image $f(u)$ of $u$ and $v$ will induce a loop in the underlying graph of $H$. 
Secondly, suppose that $u,v$ are connected by a special 2-path. Then the image $f(u)$ of $u$ and $v$ will induce multiple edges in the underlying graph of $H$. But we are working with $(m,n)$-colored mixed graphs with underlying simple graphs. Hence, $u,v$ can neither be  adjacent, nor be connected by a special 2-path.

Now assume that $u,v$ are neither   adjacent, nor  connected by a special 2-path. Then simply identify the vertices $u,v$ of $G$ to obtain another
$(m,n)$-colored mixed graph $H'$. 
Call this identified new vertex $u$ for convenience. 
If in $H'$ there are more than one edge of the same color between two vertices, delete all of them but one. 
Also,  if in $H'$ there are more than one arc of the same color in the same direction between two vertices, delete all of them but one. Doing this we obtain a new $(m,n)$-colored mixed graph $H$ from $H'$.
 Note that the underlying graph of $H$ is simple.  
 Also it is easy to see that $g: G \rightarrow H$ where $g(x) = x$ for $x \neq u,v$ and $g(u) = g(v) = u$ is a homomorphism. 
\end{proof}

\begin{corollary}
Let $G$ be an $(m,n)$-colored mixed graph. Then $G$ is a $(m,n)$-clique if and only if each pair of non-adjacent vertices of $G$ 
is connected by a special 2-path.
\end{corollary}

\begin{proof}
It follows from the definitions that $(m,n)$-colored mixed graph $G$ is a  $(m,n)$-clique if and only if  all its vertices are part of a relative clique. Then the result follows by Proposition~\ref{prop clique1}.
\end{proof}

Using these characterizations we will first prove our asymptotic result. 
	Consider  the model of generating a random $(m,n)$-mixed graph on $k$ vertices in which the adjacency type between each pair of vertices -- one of the $m$ arc types in forward or backward direction, one of the $n$ edge types or non-adjacent -- is chosen uniformly at random. We show that under such a model, as $k \rightarrow \infty$ the probability of generating an $(m,n)$-mixed clique approaches $1$.

	\begin{theorem}
	For $(m,n) \neq (0,1)$ almost every graph is an $(m,n)$-mixed clique.
\end{theorem}

\begin{proof}	Let $\mathcal{G}_k$ be the set of all $(m,n)$-mixed graphs on $k$ vertices (where $(m,n)  \neq (0,1)$), and let $\mathcal{C}_k$ be the set of all $(m,n)$-cliques on $k$ vertices. 
	
	$$|\mathcal{G}_k| = (2m+n+1)^{k \choose 2}.$$
	
	Consider generating an $(m,n)$-mixed graph, $G$, on $k$ vertices such that  a particular pair of vertices $u,v \in V(G)$ are neither adjacent nor connected by a special 2-path. 
			An easy counting argument gives that the number of such graphs is:
	
	$$ (k-2)(6m + 3n + 1)(2m+n+1)^{{k-2} \choose 2}.$$
	
	This is seen by considering the possibilities for the different sorts of adjacencies that can occur between $u,v$ and each of the other $k-2$ vertices. This implies directly the following upper bound on the cardinality $|\overline{\mathcal{C}_k}|$
	of the set of 
	$(m,n)$-mixed graph that are not $(m,n)$-cliques. 
	
	$$|\overline{\mathcal{C}_k}| < {k \choose 2}(k-2)(6m + 3n + 1)(2m+n+1)^{{k-2} \choose 2}.$$
	
	For fixed values of $m$ and $n$:
	
	$$\lim_{k \rightarrow \infty} \frac{ {k \choose 2}(k-2)(6m + 3n + 1)(2m+n+1)^{{k-2} \choose 2}}{(2m+n+1)^{k \choose 2}} = 0.$$
	
	Thus,
	
	$$ \lim_{k \rightarrow \infty} \frac{|\overline{\mathcal{C}_k}| }{|\mathcal{G}_k|} = 0.$$
	
	Therefore, as $k$ grows large, the ratio of the number of $(m,n)$-mixed cliques on $k$ vertices to the number of $(m,n)$-mixed graphs on $k$ vertices approaches $1$. 
	\end{proof}

%
%------------------------------------
%-----------------------------------
%---------------------------------
%------------------------------------

Now we will prove our result for the family $ \mathcal{O}$ of outerplanar graphs. 

\begin{theorem}\label{n-edge-colored outer clique}
Let $\mathcal{O}$ be the family of outerplanar graphs. Then $\omega_{a(m,n)}(\mathcal{O}) = \omega_{r(m,n)}(\mathcal{O}) = 3(2m+n) + 1$ for all
$(m,n) \neq (0,1)$.
\end{theorem}

\begin{proof}
First we will show that $\omega_{a(m,n)}(\mathcal{O}) \geq 3(2m+n)+1$ by explicitly constructing an outerplanar $(m,n)$-clique $H = (V, A \cup E)$ with $3(2m+n)+1$ vertices as follows:

\begin{itemize}
	\item[-] the set of vertices $V = \{x\} \cup \{v_{i,j} | 1 \leq i \leq 2m+n, 1 \leq j \leq 3\}$,
	\item[-] the set of edges $E = \{xv_{i,j} | 1 \leq i \leq n, 1 \leq j \leq 3\} \cup \{v_{i,1}v_{i,2}, v_{i,2}v_{i,3}  | 1 \leq i \leq 2m+n\}$,
	\item[-] the set of arcs $A = \{xv_{i,j}, v_{i+m,j}x | n+1 \leq i \leq n+m, 1 \leq j \leq 3\}$,
	\item[-] the edges $v_{i,1}v_{i,2}$ and $v_{i,2}v_{i,3}$ have different colors for all $i \in \{1, 2,  ..., n\}$,
	\item[-] the edges $xv_{i,j}$ recieves the $i$th edge color for all $i \in \{1, 2,  ..., n\}$ and $j \in \{1, 2, 3\}$,
	\item[-] the arcs $xv_{i+n,j}$ and $v_{i+n+m,j}x$  recieves the $i$th arc color for all $i \in \{1, 2,  ..., m\}$ and $j \in \{1, 2, 3\}$.
\end{itemize}

It is easy to check that the graph $H$ is indeed an outerplanar $(m,n)$-clique with $3(2m+n)$ vertices. Thus, 
$\omega_{r(m,n)}(\mathcal{O}) \geq \omega_{a(m,n)}(\mathcal{O}) \geq 3(2m+n) +1$

\medskip

Now to prove the upper bound let $G = (V, A \cup E)$ be a minimal (with respect to number of vertices) $(m,n)$-colored mixed graph with underlying  outerplanar graph having relative clique number 
$\omega_{r(m,n)}(\mathcal{O})$. 
Moreover, assume that  we cannot add any more edges/arcs 
keeping the graph $G$ outerplanar. We can assume this because adding more edges will not affect the relative clique number as it is already equal to 
$\omega_{r(m,n)}(\mathcal{O})$.
Let $R$ be a relative clique of cardinality $\omega_{r(m,n)}(\mathcal{O})$ of $G$. 
Let $S = V \setminus R$.

As it is not possible to increase the number of edges/arcs of $G$ keeping the graph outerplanar  $d(v) \geq 2$ for all $v \in V$. Then, as $G$ is outerplanar, there exists a vertex $u_1 \in V$ with 
$d(u_1) = 2$. 
Note that if $u_1 \in S$ then  we can delete $u_1$ and connect the neighbors of $u_1$ with an 
edge  (if they are not already adjacent) to obtain a graph with same
relative clique number contradicting the minimality  of $G$.
Thus, $u_1 \in R$. 
Fix an outerplanar embedding of $G$ with the outer (facial) cycle having vertices $u_1, u_2, ..., u_R $ of $R$ embedded 
in a clockwise manner on the cycle.
Also let  $a$ and $b$ are the two neighbors of $u_1$. Note that $a$ and $b$ are adjacent as it is not possible to increase the number of edges/arcs of $G$ keeping it outerplanar. 

% Also, assume that if you traverse on the outer facial cycle in a clockwise manner starting from $u_1$, you will find $a$ before $b$. Note that $a$ and $b$ may or may not belong to $R$. 

Every vertex of $R \setminus \{u_1,a,b\}$ is connected to $u_1$ through  $a$ or $b$ by a special 2-path. Note that, $a$ and $b$ can have at most one common neighbor other than $u_1$. 
Let that common neighbor, if it exists, be $x$. 
So, all the vertices from $R \setminus \{u_1,a,b, x\}$ are adjacent to exactly one of $a,b$. 
Also from the first part of the proof we know that $\omega_{r(m,n)}(\mathcal{O}) \geq \omega_{a(m,n)}(\mathcal{O}) \geq 3(2m+n) +1$ thus, 
$|R \setminus \{u_1,a,b,x\}| \geq 3$ for all $(m,n) \neq (0,1)$. 
Suppose neither $a$ nor $b$ is adjacent to all the vertices  of $R \setminus \{u_1,a,b, x\}$. In that case, there are two vertices 
in $R \setminus \{u_1,a,b, x\}$ that are neither adjacent nor connected by a 2-path. Therefore, either $a$ or $b$ must be adjacent to all the vertices of $R \setminus \{a,b\}$.

Assume without loss of generality that $a$ is adjacent to all the vertices of $R \setminus \{a\}$. Let 
$u_i, u_j \in (R \setminus \{a\})$ be two vertices with $i \neq j$. 
If both $u_i$ and $u_j$ are adjacent to $a$ with same color of edges, or same color of arcs with the same direction then $|i - j| \leq 2$ 
modulo $|R|$ 
as otherwise they can be neither adjacent nor connected by a special 2-path in $G$. 
Hence the number of vertices from $R \setminus \{a \}$ adjacent to $a$ with the same color of edges or same color of arcs with the same direction is at most three.

From this we can conclude that  

\begin{align}\nonumber
|R| &\leq 2\sum_{k=1}^m 3 + \sum_{k=1}^n 3 + |\{a\}| \\ \nonumber
     &\leq 3(2m+n) + 1. \nonumber
\end{align}

This completes the proof.
\end{proof}

Furthermore, using the above result, we will prove lower and upper bounds  for the $(m,n)$-absolute clique number of the family 
$ \mathcal{P}$  of planar graphs. 

\begin{theorem}
Let $\mathcal{P}$ be the family of planar graphs. Then $3(2m+n)^2+(2m+n)+1 \leq \omega_{a(m,n)}(\mathcal{P}) \leq 9(2m+n)^2 + 2(2m+n) + 2$ for all
$(m,n) \neq (0,1)$.
\end{theorem} 

\begin{proof}
First we will show that $\omega_{a(m,n)}(\mathcal{P}) \geq 3(2m+n)^2+(2m+n)+1$ by explicitly constructing a planar $(m,n)$-clique $H^* = (V^*, A^* \cup E^*)$ with $3(2m+n)^2+(2m+n)+1$ vertices. 
First recall the example of the 
outerplanar $(m,n)$-clique $H$ from the previous proof (proof of Theorem~\ref{n-edge-colored outer clique}).
Take $2m+n$ copies of $H$ and call them $H_1, H_2, ..., H_{2m+n}$. 
Let the vertices of $H_i$ be $V_i = \{v_{i,1}, v_{i,2}, ..., v_{i,3(2m+n)+1} \}$ for all $i \in \{1,2, ..., 2m+n\}$. 
The graph $H^*$ is constructed as follows:

\begin{itemize}
	\item[-] the set of vertices $V^* = \{x\} \cup (\cup_{i=1}^{2m+n} V_i)$,
	\item[-] the edges $xv_{i,j}$ has the $i$th edge color for $i \in \{1,2,...,n\}$ and $j \in \{1,2,...,3(2m+n)+1\}$,
	\item[-] the arcs $xv_{i+n,j}$ and $v_{i+n+m,j}x$ has the $ith$ arc color for $i \in \{1,2,...,m\}$ and $j \in \{1,2,...,3(2m+n)+1\}$. 
\end{itemize}

It is easy to check that the graph $H^*$ is indeed a planar $(m,n)$-clique  with $3(2m+n)^2+(2m+n)+1$ vertices.

\medskip

Now to prove the upper bound first notice that any $(m,n)$-clique has diameter at most 2. 
Let $G = (V, A \cup E)$ be 
a planar $(m,n)$-clique with more than $\omega_{a(m,n)}(G) > 3(2m+n)^2+(2m+n)+1$. Assume that $G$ is triangulated. 
As deleting edges do not increase the $(m,n)$-absolute clique number, it is enough to prove this result for triangulated $G$. 

Suppose that the set of vertices adjacent to a vertex $u$ by an edge with $i$th edge color is $N_i(u)$ for all $i \in \{1,2,...,n\}$. 
Also, let the set of vertices $v$ that are adjacent to $u$ with an arc $uv$ with $j$th arc color is $N_j^+(u)$ and the set of vertices $v$ that are adjacent to $u$ with an arc $vu$ with $j$th arc color is $N_j^-(u)$ for $j \in \{1,2,...,m\}$.

We know that a diameter two planar graph is dominated by at most two vertices except for a particular graph on nine vertices due to Goddard 
and Henning~\cite{dom}. We have already shown in the first part of the proof that 
$\omega_{a(m,n)}(\mathcal{P}) \geq 3(2m+n)^2+(2m+n)+1 \geq 15$ for   all $(m,n) \neq (0,1)$. Thus,  we can assume that $G$ is dominated by at most two vertices. 

First assume that $G$ is dominated by a single vertex $x$. 
Let $G'$ be the graph obtained by deleting the vertex $x$ from $G$. 
Note that $G'$ is an outerplanar graph. Furthermore, the graph induced by $N_i(x)$  is a relative $(m,n)$-clique of $G'$. 
Thus, by Theorem~\ref{n-edge-colored outer clique} $|N_i(x)| \leq 3(2m+n)+1$ for all $i \in \{1,2,...,n\}$. 
Similarly, $|N_j^+(x)|, |N_j^-(x)| \leq 3(2m+n)+1$ for all $j \in \{1,2,...,m\}$. 
Thus,

\begin{align}\nonumber
\omega_{an}(G) & \leq |\cup_{i=1}^n N_i(x)| + |\cup_{j=1}^m N_j^+(x)| + |\cup_{j=1}^m N_j^-(x)| +|\{x\}| \\ \nonumber
    & \leq 3(2m+n)^2+(2m+n)+1. \nonumber
\end{align}

Now let $G$ has domination number 2. 
Let $x,y$ be such that they dominates $G$ and has maximum number of common neighbors among all pairs of dominating vertices of $G$. 
Now we fix some notations to prove the rest of this result.

\begin{itemize}
	\item[-] $C = N(x) \cap N(y)$ and $C_{ij} = N_i(x) \cap N_j(y)$ for all $i,j \in \{1,2,...,n\}$,
	\item[-] $C_{ij}^{* \alpha} = N_i(x) \cap N_j^{\alpha}(y)$ and $C_{ji}^{ \alpha *} = N_j^{\alpha}(x) \cap N_i(y)$ for all 
	$i \in \{1,2,...,n\}$, $j \in \{1,2,...,m\}$ and $\alpha \in \{+,-\}$, 
	\item[-] $C_{ij}^{\alpha \beta} = N_i^{\alpha}(x) \cap N_j^{\beta}(y)$  for all 	$i, j \in \{1,2,...,m\}$ and $\alpha, \beta \in \{+,-\}$,
	\item[-] $S_{x_i} = N_i(x) \setminus C$, $S_{y_i} = N_i(y) \setminus C$, $S_{x_j}^{\alpha} = N_j^{\alpha}(x) \setminus C$ and $S_{y_j}^{\alpha} = N_j^{\alpha}(y) \setminus C$ for all $i \in \{1,2,...,n\}$, $j \in \{1,2,...,m\}$ and $\alpha \in \{+,-\}$,
	\item[-] $S_x = N(x) \setminus C$, $S_y = N(y) \setminus C$ and $S = S_x \cup S_y$. 
\end{itemize}

Note that, for $|C| \geq 6$ we must have $|C_{ij}|, |C_{ik}^{* \beta}|, |C_{lj}^{ \alpha*}|, |C_{lk}^{\alpha \beta}| \leq 3$ for all $i,j \in \{1,2,...,n\}$, 
$l,k \in \{1,2,...,m\}$ and $\alpha, \beta \in \{+,-\}$ as otherwise it is not possible to have 
pairwise distance at most two between  the  vertices of $C$ which has the same incidence rule with both $x$ and $y$
 keeping the graph planar. So, we can conclude that

\begin{align}\nonumber
|C| \leq  \left| \bigcup_{1 \leq i,j \leq n} C_{ij} \right| + \left|\bigcup_{\substack{1 \leq i \leq n,\\
                                           1 \leq k \leq m,\\
                                           \alpha \in \{+,-\} }} C_{ik}^{* \beta}\right| + 
                                           \left|\bigcup_{\substack{1 \leq l \leq m,\\
                                          1 \leq j \leq n,\\
                                          \alpha \in \{+,-\} }} C_{lj}^{ \alpha *}\right| + 
                                          \left|\bigcup_{\substack{1 \leq l,k \leq m,\\
                                          \alpha, \beta \in \{+,-\} }} C_{lk}^{\alpha \beta}\right|     \leq 3(2m+n)^2. \nonumber
\end{align}

As $3(2m+n)^2 \geq 6$ for all $(m,n) \neq (0,1)$ we can say that $|C| \leq 3(2m+n)^2$ in general. 

\medskip

If $|C| \geq 2$ then pick two vertices $u,v \in C$. Note that the cycle induced by $x,y,u,v$ divides the plane into two regions: denote the interior by $R_1$ and the exterior by $R_2$. Consider a  planar embedding of $G$. Observe that if we delete the vertices $x,y$ and all the vertices placed in $R_2$, then the resultant graph, denoted by $G_1$, is outerplanar.  Similarly, if we delete the vertices $x,y$ and all the vertices placed in $R_1$, then the resultant graph, denoted by $G_2$, is outerplanar.

Observe that if we restrict ourselves to $G_1$,  then the set $[S_{x_i} \cup S_{y_i}]_{G_1}$ of vertices of $(S_{x_i} \cup S_{y_i})$ which are also part of $G_1$  is a relative clique. Thus,
 by Theorem~\ref{n-edge-colored outer clique} $[S_{x_i} \cup S_{y_i}]_{G_1} \leq 3(2m+n) + 1$ for all $i \in \{1,2,...,n\}$. 
 Similarly, we can show that $[S_{x_j}^{\alpha} \cup S_{y_j}^{\alpha}]_{G_1} \leq 3(2m+n) +1$ for all $j \in \{1,2,...,m\}$ and $\alpha \in \{+,-\}$. Thus, 
 in $G_1$ we have $| [S]_{G_1} |\leq 3(2m+n)^2 + (2m+n)$.

 Similarly, in $G_2$ we have $| [S]_{G_2} | \leq 3(2m+n)^2 + (2m+n)$. Thus, in $G$ we have $|S| \leq 6(2m+n)^2 + 2(2m+n)$. 
 
 \medskip
 
 For $|C| = 1$, the graph obtained by deleting the vertices $x$ and $y$ is outerplanar. So, repeating the same argument as above we get 
 $|S| \leq 3(2m+n)^2 + (2m+n)$ in this case. Hence, no matter what, we have $|S| \leq 6(2m+n)^2 + 2(2m+n)$.

Therefore,
\begin{align}\nonumber
|G| = |C| + |S| + |\{x,y\}|   \leq 9(2m+n)^2 + 2(2m+n) +2. \nonumber
\end{align}

Hence we are done.
\end{proof}

Finally, we would like to make the following conjecture regarding the $(m,n)$-absolute clique number of planar graphs.

\begin{conjecture}
For the family $\mathcal{P}$ of planar graphs   $\omega_{a(m,n)}(\mathcal{P}) = 3(2m+n)^2+(2m+n)+1$ for all $(m,n) \neq (0,1)$.
\end{conjecture}

It is known that the conjecture is true for $(m,n) = (1,0)$~\cite{unique_oclique} and $(0,2)$~\cite{myphd}.

\section{Complexity aspects}\label{complexity}
We know that the complexity of deciding whether, given an undirected simple graph $G$, we can assign colors to the arcs and edges of $G$ to make it an $(m,n)$-clique is \NPclass-hard for $(m,n) = (1,0)$ and $(0,2)$~\cite{oclique-complexity}. Here we address a new related problem concerning signed graphs, an equivalence class of $(0,2)$-colored mixed graphs~\cite{signedhom}. 
We will not discuss signed graphs  here but will 
 encourage the reader to have a look at~\cite{signedhom} instead for finding motivation of the problem. 
 We will formulate the problem in 
terms of the definitions and notions given in this article to keep this paper self-contained.

In a $(0,2)$-colored mixed graph $G$, by an \textit{unbalanced 4-cycle} we refer to a $4$-cycle of $G$ having an odd number of edges of the same color. 
We call $G$ a \textit{signed clique} if each pair of vertices of $G$ is part of an unbalanced 4-cycle. 
An undirected simple graph $G$ is being \textit{$2$-edge-colored} if we assign a color to each of its edges from a fixed set of two colors.  

Our main result of this section reads as follows: It is \NPclass-complete to decide whether a given undirected graph can be $2$-edge-colored so that we obtain a signed clique. In other words, there should not be an easy characterization of signed cliques in terms of their underlying undirected graphs.

\bigskip

\noindent \SSCLIQUEPB \\
\noindent Input: An undirected graph $G$. \\
\noindent Question: Can $G$ be $2$-edge-colored so that we get a signed clique?

\medskip

\begin{theorem}
\SSCLIQUEPB~is \NPclass-complete.
\end{theorem}

\begin{proof}
We prove the \NPclass-hardness of \SSCLIQUEPB~by reduction from the following \NPclass-complete problem.

\medskip

\noindent \MONNAESATPB \\
\noindent Instance: A $3$CNF formula $F$ over variables $x_1, x_2, ..., x_n$ and clauses $C_1, C_2, ..., C_m$ involving no negated variables. \\
\noindent Question: Is $F$ \textit{not-all-equal satisfiable}, that is, does there exist a truth assignment to the variables under which every clause has at least one true and one false variable?

\medskip

 Due to the \NPclass-completeness of the {\scshape 2-Colouring of $3$-Uniform Hypergraph} problem (see~\cite{Lov73}), it is easily seen that \MONNAESATPB~remains \NPclass-complete when every clause of $F$ has its three variables being different. So this additional restriction is understood throughout. From a $3$CNF formula $F$, we construct an undirected graph $G_F$ such that

\begin{center}
$F$ is not-all-equal satisfiable \\
$\Leftrightarrow$ \\
$G_F$ can be $2$-edge-colored in a signed clique way.
\end{center}

The construction of $G_F$ is achieved in two steps. We first construct, from $F$, an undirected graph $H_F$ such that $F$ is not-all-equal satisfiable if and only if there exists a $2$-edge-coloring $c_H$ of $H$ under which only some \textit{representative pairs} of non-adjacent vertices belong to unbalanced $4$-cycles. This equivalence is obtained by designing $H_F$ in such a way that every representative pair belongs to a unique $4$-cycle, with 
some of these $4$-cycles overlapping to force some edges to have the same or different colors by $c_H$. Then we obtain $G_F$ by adding some vertices and edges to $H_F$ in such a way that no new $4$-cycles including representative pairs are created, and there exists a partial $2$-edge-coloring of the edges in $E(G_F) \setminus E(H_F)$ for which every non-representative pair is included in an unbalanced $4$-cycle. In this way, the equivalence between $G_F$ and $F$ is only dependent of the equivalence between $H_F$ and $F$, which has not been altered when constructing $G_F$ from $H_F$.

\bigskip

 \textbf{Step~1.} Start by adding two vertices $r_1$ and $r_2$ to $H_F$. Then, for every variable $x_i$ of $F$, add two vertices $u_i$ and $u'_i$ to $H_F$, and link these vertices to both $r_1$ and $r_2$. Now, for every $i \in \{1, 2, ..., n\}$, assuming the variable $x_i$ belongs to the (distinct) clauses $C_{j_1}, C_{j_2}, ..., C_{j_{n_i}}$, add $n_i$ new vertices $v_{i,j_1}, v_{i,j_2}, ..., v_{i, j_{n_i}}$ to $H_F$, and join all these new vertices to both $r_1$ and $r_2$. Finally, for every clause $C_j = (x_{i_1} \vee x_{i_2} \vee x_{i_3})$ of $F$, add a new vertex $w_j$ to $H_F$, and join it to all of $v_{i_1,j}, v_{i_2,j}, v_{i_3,j}$. 

The representative pairs are the following. For every variable $x_i$ of $F$, all pairs $\{u_i, u'_i\}$ and those of the form $\{u'_i, v_{i,j}\}$ are representative. Also, for every clause $C_j = (x_{i_1} \vee x_{i_2} \vee x_{i_3})$ of $F$, the pairs $\{v_{i_1,j}, v_{i_2,j}\}$, $\{v_{i_1,j}, v_{i_3,j}\}$ and $\{v_{i_2,j}, v_{i_3,j}\}$ are representative.

\medskip

We below prove some claims about the existence of a \textit{good $2$-edge-coloring} $c_H$ of $H_F$, \textit{i.e.} a $2$-edge-coloring under which every representative pair of $H_F$ belongs to an unbalanced $4$-cycle. By $(H_F, c_H)$, we refer to the $2$-edge-colored graph obtained from $H_F$ by coloring its edges as indicated by $c_H$. Given two edges $u_1u$ and $u_2u$ of $H_F$ meeting at $u$, we say below that $u_1$ and $u_2$ \textit{agree} (resp. \textit{disagree}) on $u$ (by $c_H$) if $u_1u$ and $u_2u$ are assigned the same color (resp. distinct colors) by $c_H$.

\begin{claim} \label{claim1}
Let $c_H$ be a good $2$-edge-coloring of $H_F$, and let $x_i$ be a variable appearing in clauses $C_{j_1}, C_{j_2}, ..., C_{j_{n_i}}$ of $F$. If $r_1$ and $r_2$ agree (resp. disagree) on $u_i$, then $r_1$ and $r_2$ agree (resp. disagree) on $v_{i, j_1}, v_{i, j_2}, ..., v_{i, j_{n_i}}$.
\end{claim}

\begin{proof}
Because $\{u_i, u'_i\}$ is representative and $u_ir_1u'_ir_2u_i$ is the only $4$-cycle containing $u_i$ and $u'_i$, the vertices $r_1$ and $r_2$ agree on $u_i$ and disagree on $u'_i$ in $(H_F, c_H)$ without loss of generality. Now, because every pair $\{u'_i, v_{i, j}\}$ is representative, the only $4$-cycle including $u'_i$ and $v_{i ,j}$ is $u'_ir_1v_{i,j}r_2u'_i$, and $r_1$ and $r_2$ disagree on $u'_i$ in $(H_F, c_H)$, necessarily $r_1$ and $r_2$ agree on $v_{i,j}$.
\end{proof}

\begin{claim} \label{claim2}
Let $c_H$ be a good $2$-edge-coloring of $H_F$, and let $C_j = (x_{i_1} \vee x_{i_2} \vee x_{i_3})$ be a clause of $F$. Then $r_1$ and $r_2$ cannot agree or disagree on all of $v_{i_1,j}, v_{i_2,j}, v_{i_3,j}$.
\end{claim}

\begin{proof}
First note that the only $4$-cycles of $H_F$ containing, say, $v_{i_1,j}$ and $v_{i_2,j}$ are $v_{i_1,j}r_1v_{i_2,j}r_2v_{i_1,j}$, $v_{i_1,j}w_jv_{i_2,j}r_1v_{i_1,j}$ and $v_{i_1,j}w_jv_{i_2,j}r_2v_{i_1,j}$. The claim then follows from the fact that if $r_1$ and $r_2$, say, agree on all of $v_{i_1,j}, v_{i_2,j}, v_{i_3,j}$, then $(H_F, c_H)$ has no unbalanced $4$-cycle including $r_1$ and $r_2$ and two of $v_{i_1,j}, v_{i_2,j}, v_{i_3,j}$. So, since $c_H$ is a good $2$-edge-coloring, necessarily there are at least three unbalanced $4$-cycles containing $w_j$ and every two of $v_{i_1,j}, v_{i_2,j}, v_{i_3,j}$, but one can easily convince himself that this is impossible.

Assume on the contrary that \textit{e.g.} $r_1$ and $r_2$ agree on $v_{i_1,j}$ and disagree on $v_{i_2,j}$ and $v_{i_3,j}$. So far, note that $r_1v_{i_1,j}r_2v_{i_2,j}r_1$ and $r_1v_{i_1,j}r_2v_{i_3,j}r_1$ are unbalanced $4$-cycles of $(H_F, c_H)$. Then there is no contradiction against the fact that $c_H$ is good, since \textit{e.g.} $v_{i_2,j}$ and $v_{i_3,j}$ can agree on $w_j$ (and, in such a situation, $v_{i_2,j}w_jv_{i_3,j}r_1v_{i_2,j}$ is an unbalanced $4$-cycle). The important thing to have in mind is that coloring the edges incident to $w_j$ can only create unbalanced $4$-cycles containing the representative pairs $\{v_{i_1,j}, v_{i_2,j}\}$, $\{v_{i_1,j}, v_{i_3,j}\}$ and $\{v_{i_2,j}, v_{i_3,j}\}$. So coloring the edges incident to $w_j$ to make $v_{i_2,j}$ and $v_{i_3,j}$ belong to some unbalanced $4$-cycle does not compromise the existence of other unbalanced $4$-cycles including farther vertices from another representative pair.
\end{proof}

We claim that we have the desired equivalence between not-all-equal satisfying $F$ and finding a good $2$-edge-coloring $c_H$ of $H_F$. To see this holds, just assume, for every variable $x_i$ of $F$, that having $r_1$ and $r_2$ agreeing (resp. disagreeing) on $u_i$ simulates the fact that variable $x_i$ of $F$ is set to \textit{true} (resp. \textit{false}) by some truth assignment, and that having $r_1$ and $r_2$ agreeing (resp. disagreeing) on some vertex $v_{i, j}$ simulates the fact that variable $x_i$ provides value \textit{true} (resp. \textit{false}) to the clause $C_j$ of $F$. Then the property described in Claim~\ref{claim1} depicts the fact that if $x_i$ is set to some truth value by a truth assignment, then $x_i$ provides the same truth value to every clause containing it. The property described in Claim~\ref{claim2} depicts the fact that every clause $C_j$ is considered satisfied by some truth assignment if and only if $C_j$ is supplied different truth values by its variables. So we can deduce a good $2$-edge-coloring of $H_F$ from a truth assignment not-all-equal satisfying $F$, and vice-versa.

\medskip

\textbf{Step~2.} As described above, we now construct $G_F$ from $H_F$ in such a way that 

\begin{itemize}
	\item every $4$-cycle of $G_F$ including the vertices of a representative pair is also a $4$-cycle in $H_F$,
	
	\item the edges of $E(G_F) \setminus E(H_F)$ can be colored with two colors so that every two vertices which do not form a representative pair belong to an unbalanced $4$-cycle.
\end{itemize}

\noindent In this way, $G_F$ will be the support of a signed clique if and only if $H_F$ admits a good $2$-edge-coloring, which is true if and only if $F$ can be not-all-equal satisfied. The result will then hold by transitivity.

\medskip

For every vertex $u$ of $H_F$, add the edges $ua_u$ and $ub_u$, where $a_u$ and $b_u$ are two new vertices. Now, for every two distinct vertices $u$ and $v$ of $H_F$, if $\{u,v\}$ is not a representative pair, then add another vertex $c_{u,v}$ to the graph, as well as the edges $uc_{u,v}$ and $vc_{u,v}$. Finally turn the subgraph induced by all newly added vertices into a clique. The resulting graph is $G_F$. As claimed above, note that the only $4$-cycles of $G_F$ containing two vertices $u$ and $v$ forming a representative pair are those of $H_F$. Every other such new cycle has indeed length at least~$6$. Namely, every such new cycle starts from $u$, then has to enter the clique by either $a_u$ or $b_u$, cross the clique to either $a_v$ or $b_v$, reach $v$, before finally going back to $u$.

\medskip

Consider the following coloring with two colors of the edges in $E(G_F) \setminus E(H_F)$. For every vertex $u \in V(H_F)$, let $ub_u$ be colored~$1$. Similarly, for every two distinct vertices $u$ and $v$ of $H_F$ such that $\{u ,v\}$ is not representative (\textit{i.e.} $c_{u,v}$ exists), let $c_{u,v}v$ be colored~$1$. Let finally all other edges be colored~$2$. Clearly, under this partial $2$-edge-coloring of $G_F$, every two vertices $u$ and $v$ of $G_F$ not forming a representative pair are either adjacent or belong to some unbalanced $4$-cycle:

\begin{itemize}
	\item if $u$ and $v$ do not belong to $H_F$, then they belong to the clique and are hence adjacent;
	\item if $u$ belongs to $H_F$ but $v$ does not, then observe that either $u$ and $v$ are adjacent (in this situation $v$ is either $a_u$, $b_u$ or $c_{u,i}$ for some $i$), or $ua_uvb_uu$ is an unbalanced $4$-cycle;
	\item if $u$ and $v$ are vertices of $H_F$ and $\{u,v\}$ is not representative, then \textit{e.g.} $uc_{u,v}va_va_uu$ is an unbalanced $4$-cycle.
\end{itemize}

\medskip

According to all previous arguments, finding a truth assignment not-all-equal satisfying $F$ is equivalent to finding a signed clique from $G_F$, as claimed. So \SSCLIQUEPB~is \NPclass-hard, and hence \NPclass-complete.
\end{proof}

\section{Comparing mixed chromatic number of $(1,0)$-colored and $(0,2)$-colored mixed graphs}\label{o-vs-s}
This comparison is motivated from the fact that for several graph classes the two parameters $\chi_{(1,0)}$  and $\chi_{(0,2)}$
 have the same or similar bounds, some of them tight~\cite{sopena_updated_survey,homsignified}. In most cases the same or similar
 bounds are proved using similar techniques. The similarity is probably due to the fact that in both cases adjacency can be of two types. 
 Thus, there was a reasonable amount of hints indicating an underlying relation between the two parameters. Here we end up proving the opposite, that is, the parameters can  be arbitrarily different. 

\begin{theorem}\label{theorem difference}
Given any integer $n$, there exists an undirected graph $G$ such that 
$\chi_{(0,2)}(G) - \chi_{(1,0)}(G) = n$.
\end{theorem}

\begin{proof}
All the complete graphs have their oriented chromatic number equal to their signed chromatic number. So, we need to prove $\chi_{(0,2)}(G) - \chi_{(1,0)}(G) = n$ for non-zero integers $n$. 

Let $A$ and $B$ be two undirected graphs. Let $A+B$ be the undirected graph obtained by 
taking disjoint copies of $A$ and $B$ and adding a new vertex $\infty$  adjacent to all the vertices of $A$ and $B$. So,  $A+B$ is the undirected graph dominated by the vertex $\infty$ and $N(\infty)$ is the disjoint union of $A$ and $B$. 

It is easy to observe that 
\begin{align}\nonumber
\chi_{(0,2)}(A+B) \leq \chi_{(0,2)}(A) + \chi_{(0,2)}(B) +1.
\end{align}

Let the edges connecting $\infty$ and the vertices of $A$ recieve the first color. Let the edges connecting $\infty$ and the vertices of $B$ recieve the second color. 
Also retain the edge-colors of the edges of $A$ and $B$ in $A+B$ for which they attain their respective chromatic numbers. Note that we have to use different set of colors for coloring the vertices of $A$ and the vertices of $B$ in this sort of color assignment of edges. Also, we need to use 
at least $\chi_{(0,2)}(A)$ colors to color the vertices of $A$ and at least  $\chi_{(0,2)}(B)$ colors to color the vertices of $B$.
Furthermore, the vertex $\infty$ must recieve a color different from the colors used for any other vertices of $A+B$. Thus, 
\begin{align}\nonumber
\chi_{(0,2)}(A+B) = \chi_{(0,2)}(A) + \chi_{(0,2)}(B) +1.
\end{align}

\medskip

Similarly, consider the orientations $\overrightarrow{A}$, $\overrightarrow{B}$ for which we have 
$\chi_{(1,0)}(\overrightarrow{A}) = \chi_{(1,0)}(A)$ and  $\chi_{(1,0)}(\overrightarrow{B}) = \chi_{(1,0)}(B)$.
Now orient  the edges of $A$ and $B$ in $A+B$ like in $\overrightarrow{A}$ and $\overrightarrow{B}$.
Now orient the edges incedent to the vertex $\infty$ in such a way that we have 
the arcs  $\infty a$ for all $a \in A$ and $b \infty $ for all $b \in B$.  
Note that we have to use different set of colors for coloring the vertices of $A$ and the vertices of $B$ for this sort of an orientation.
Thus,
\begin{align}\label{eqn oriented chi}
\chi_{(1,0)}(A+B) = \chi_{(1,0)}(A) + \chi_{(1,0)}(B) +1.
\end{align}

Let $H$ be an undirected graph. Then we define, by induction, the graph $H_k = H + H_{k-1}$ for 
$k \geq 2$ where $H_1=H$. Note that,

\begin{align}\nonumber
\chi_{(0,2)}(H_k) = k \times \chi_{(0,2)}(H) + (k-1) \text{ and } \chi_{(1,0)}(H_k) = k \times \chi_{(1,0)}(H) + (k-1).
\end{align}

The above two equations imply

\begin{align}\label{eqn signed-oriented}
\chi_{(0,2)}(H_k) - \chi_{(1,0)}(H_k) = k\times(\chi_{(0,2)}(H) - \chi_{(1,0)}(H)).
\end{align}

It is easy to observe that, a path $P_5$ of length 5 (number of edges) has $\chi_{(0,2)}(P_5) = 4$ and $\chi_{(1,0)}(P_5) = 3$.
On the other hand,  a cycle $C_5$ of length 5 has $\chi_{(0,2)}(C_5) = 4$ and $\chi_{(1,0)}(C_5) = 5$.
Hence we have  $\chi_{(0,2)}(P_5) - \chi_{(1,0)}(P_5) = 1$ and $\chi_{(0,2)}(C_5) - \chi_{(1,0)}(C_5) = -1$.

Now by replacing $H$ with $P_5$ and $C_5$ in equation~\ref{eqn signed-oriented} we are done.
\end{proof}

This result shows that the two parameters does not have any additive relation between them. Even though we are not ruling out the possibility of other types of relation between the two parameters. For exapmle, one can try to figure out if the ratio  
$\chi_{(0,2)}(G)/\chi_{(1,0)}(G)$ is bounded or not where $G$ is an undirected simple graph.

\bibliographystyle{abbrv}
\bibliography{NSS14}

\begin{thebibliography}{10}

\bibitem{Marshall-edgecoloring}
N.~Alon and T.~H. Marshall.
\newblock Homomorphisms of edge-colored graphs and {C}oxeter groups.
\newblock {\em Journal of Algebraic Combinatorics}, 8(1):5--13, 1998.

\bibitem{oclique-complexity}
J.~Bensmail, R.~Duvignau, and S.~Kirgizov.
\newblock The complexity of deciding whether a graph admits an orientation with
  fixed weak diameter.
\newblock Preprint available at http://hal.archives-ouvertes.fr/hal-00824250,
  2013.

\bibitem{Borodinacyclic}
O.~V. Borodin.
\newblock On acyclic colorings of planar graphs.
\newblock {\em Discrete Mathematics}, 25(3):211--236, 1979.

\bibitem{dom}
W.~Goddard and M.~A. Henning.
\newblock Domination in planar graphs with small diameter.
\newblock {\em Journal of Graph Theory}, 40:1--25, May 2002.

\bibitem{36}
W.~F. Klostermeyer and G.~MacGillivray.
\newblock Analogs of cliques for oriented coloring.
\newblock {\em Discussiones Mathematicae Graph Theory}, 24(3):373--388, 2004.

\bibitem{Lov73}
L.~Lov\'{a}sz.
\newblock Coverings and coloring of hypergraphs.
\newblock {\em Proceedings of the Fourth South-eastern Conference on
  Combinatorics, Graph Theory, and Computing, Boca Raton, Florida}, page
  3–12, 1973.

\bibitem{homsignified}
A.~Montejano, P.~Ochem, A.~Pinlou, A.~Raspaud, and {\'E}.~Sopena.
\newblock Homomorphisms of 2-edge-colored graphs.
\newblock {\em Discrete Applied Mathematics}, 158(12):1365--1379, 2010.

\bibitem{unique_oclique}
A.~Nandy, S.~Sen, and {\'E}.~Sopena.
\newblock Outerplanar and planar oriented cliques (accepted).
\newblock {\em Journal of Graph Theory}, 2015.

\bibitem{signedhom}
R.~Naserasr, E.~Rollov{\'a}, and {\'E}.~Sopena.
\newblock Homomorphisms of signed graphs. (preprint).
\newblock 2013.

\bibitem{raspaud_and_nesetril}
J.~Ne{\v{s}}et{\v{r}}il and A.~Raspaud.
\newblock Colored homomorphisms of colored mixed graphs.
\newblock {\em Journal of Combinatorial Theory, Series B}, 80(1):147--155,
  2000.

\bibitem{planar80}
A.~Raspaud and {\'E}.~Sopena.
\newblock Good and semi-strong colorings of oriented planar graphs.
\newblock {\em Information Processing Letters}, 51(4):171--174, 1994.

\bibitem{myphd}
S.~Sen.
\newblock {\em {A contribution to the theory of graph homomorphisms and
  colorings}}.
\newblock PhD thesis, University of Bordeaux, France, 2014.

\bibitem{sopena_updated_survey}
E.~Sopena.
\newblock Homomorphisms and colourings of oriented graphs: An updated survey.
\newblock {\em Discrete Mathematics}, (0):--, 2015.

\end{thebibliography}

\end{document}